\documentclass[11pt]{article}

\usepackage[a4paper,margin=1in]{geometry}
\usepackage{amsmath,amssymb,amsthm,mathtools}
\usepackage{enumitem}
\usepackage{microtype}
\usepackage[hidelinks]{hyperref}

\newtheorem{theorem}{Theorem}[section]
\newtheorem{proposition}[theorem]{Proposition}
\newtheorem{lemma}[theorem]{Lemma}

\newtheorem{remark}[theorem]{Remark}

\newcommand{\N}{\mathbb N}
\newcommand{\Z}{\mathbb Z}
\newcommand{\cX}{\mathcal X}
\newcommand{\cE}{\mathcal E}
\newcommand{\cP}{\mathcal P}
\newcommand{\Prob}{\mathbb P}
\newcommand{\one}{\mathbf 1}
\newcommand{\io}{\mathrm{i.o.}}
\newcommand{\law}{\mathcal L}

\title{Weak but Not Strong Asymptotic Testability}
\author{Senhan Yao}
\date{July 29, 2026}

\begin{document}
\maketitle

\begin{abstract}
We construct two fixed disjoint families $H_0,H_1$ of stationary ergodic binary process distributions for which a weakly asymptotically consistent test exists, but no strongly asymptotically consistent test exists.  The construction combines a synchronizing binary suspension code, an independent i.i.d.\ marker process, and countably many independent slowly switching two-state Markov chains. In particular, this disproves the asymptotic-consistency branch of a conjecture of Ryabko.
\end{abstract}

\section{Definitions and main result}\label{sec:definitions}

Throughout,
\[
  \N:=\{1,2,3,\ldots\},
  \qquad
  \N_0:=\{0,1,2,\ldots\}.
\]
Let
\[
  \cX:=\{0,1\}^{\N},
  \qquad
  \overline{\cX}:=\{0,1\}^{\Z},
\]
each equipped with its product Borel sigma-field.  Let
$T:\cX\to\cX$ and $\overline T:\overline{\cX}\to\overline{\cX}$ be
the left shifts,
\[
  (Tx)_i=x_{i+1}\quad(i\in\N),
  \qquad
  (\overline T x)_i=x_{i+1}\quad(i\in\Z).
\]
Let $\cE$ denote the set of all $T$-invariant and $T$-ergodic Borel
probability measures on $\cX$.  Here a $T$-invariant probability
measure $P$ is called $T$-ergodic when every Borel set $A$ satisfying
$P(A\mathbin\triangle T^{-1}A)=0$ has $P(A)\in\{0,1\}$.  For
$x=(x_1,x_2,\ldots)\in\cX$, write $x_1^n=(x_1,\ldots,x_n)$.

We use the following standard terminology throughout.  A
\emph{probability-preserving system} is a quadruple
$(\Omega,\mathcal F,\mu,S)$ in which $(\Omega,\mathcal F,\mu)$ is a
probability space, $S:\Omega\to\Omega$ is measurable, and
$\mu(S^{-1}A)=\mu(A)$ for every $A\in\mathcal F$.  It is
\emph{mixing} if
\[
  \mu(A\cap S^{-m}B)\longrightarrow\mu(A)\mu(B)
  \qquad(m\to\infty)
\]
for all fixed $A,B\in\mathcal F$.  Throughout this paper, the word
``mixing'' refers only to this measure-theoretic two-set condition for a
single transformation.  It does not mean Rosenblatt (or strong)
$\alpha$-mixing, which takes a supremum over past and future
sigma-fields; no mixing coefficient and no rate of mixing is assumed.
A measurable map
$\Phi:(\Omega,\mathcal F,\mu,S)\to(\Omega',\mathcal F',\nu,R)$ is a
\emph{factor map} if $\nu=\Phi_\#\mu$ and
$\Phi\circ S=R\circ\Phi$ almost surely.  Here the pushforward is
defined by
\[
  (\Phi_\#\mu)(C):=\mu(\Phi^{-1}C),\qquad C\in\mathcal F'.
\]

A test is a sequence of measurable functions
\[
  \varphi_n:\{0,1\}^n\longrightarrow\{0,1\},\qquad n\in\N.
\]
Let $H_0,H_1\subset\cE$ be disjoint; they need not exhaust $\cE$.
No measurability assumption on $H_0$ or $H_1$ as a subset of the space
of process laws is required by the definitions below.
The test $(\varphi_n)$ is \emph{weakly asymptotically consistent} for
$H_0$ against $H_1$ if, for every $i\in\{0,1\}$ and every
$\rho\in H_i$,
\[
  \rho\{x:\varphi_n(x_1^n)\ne i\}\longrightarrow0.
\]
It is \emph{strongly asymptotically consistent} if, for every
$i\in\{0,1\}$ and every $\rho\in H_i$,
\[
  \rho\!\left\{x:\exists N\in\N\ \forall n\ge N,
  \ \varphi_n(x_1^n)=i\right\}=1.
\]
The event in the preceding display is measurable because it equals
$\bigcup_{N\ge1}\bigcap_{n\ge N}
\{x:\varphi_n(x_1^n)=i\}$.  Because the test values are binary, the
condition is equivalent to $\varphi_n(X_1^n)\to i$ almost surely under
every $\rho\in H_i$.  For any process law $P$, test $\psi$, and
$i\in\{0,1\}$, we use the shorthand
\[
  P\{\psi_n=i\}:=P\{x:\psi_n(x_1^n)=i\}.
\]
No convergence rate and no uniformity over either hypothesis is
required.  These are precisely the weak and strong asymptotic notions
relevant to the asymptotic part of Ryabko's Conjecture~5.1
\cite{Ryabko2019}.

\begin{remark}[Scope of the counterexample]\label{rem:scope}
The implication refuted in this paper is the unrestricted assertion
\[
 \text{``for every pair of disjoint sets }H_0,H_1\subset\cE,
 \text{ weak consistency implies strong consistency.''}
\]
Accordingly, the definitions above impose no Borel, closedness, or
topological condition on the hypotheses as subsets of a space of
probability laws.  The theorem below makes no claim about variants of
the implication in which additional regularity assumptions are imposed
on $H_0$ and $H_1$.  This scope condition is part of the statement being
proved, not an unmentioned convention.
\end{remark}

\begin{theorem}[Main theorem]\label{thm:main}
There exist fixed disjoint sets $H_0,H_1\subset\cE$ such that
\begin{enumerate}[label=\textup{(\roman*)}]
\item a weakly asymptotically consistent test exists for $H_0$ against
      $H_1$;
\item no strongly asymptotically consistent test exists for $H_0$
      against $H_1$.
\end{enumerate}
Consequently, weak asymptotic consistency does not imply strong
asymptotic consistency for arbitrary pairs of stationary ergodic
hypotheses.
\end{theorem}

\section{Measure-theoretic conventions and auxiliary results}
\label{sec:standard-results}

All products below carry their product sigma-fields, and all subsets of
suspension spaces carry the corresponding trace sigma-fields.  Random
variables defined on product spaces are understood through the relevant
coordinate projections.  Equalities and invariance statements involving
measurable sets are interpreted modulo null sets unless exact equality is
explicitly asserted.

The underlying measure theory is ordinary countably additive probability.
Every measurable space used below is equipped explicitly with its
sigma-field: finite or countable discrete spaces carry their full
sigma-fields, sequence spaces carry product sigma-fields, and suspension
spaces carry trace sigma-fields.  No compactness theorem, regular
conditional-probability theorem, or other unlisted topological
regularity result is used.

We take the following elementary measure-theoretic facts as foundational:
the construction and uniqueness clause of the Carath\'eodory extension
theorem; existence of finite product measures; continuity from above and
below for finite measures; the monotone convergence theorem; and
Tonelli--Fubini for nonnegative or integrable functions.  In particular,
if two finite measures agree on an algebra that generates the ambient
sigma-field, the uniqueness clause of the extension theorem implies that
they agree on the generated sigma-field.  Whenever a countable product
law is used below, its existence is either constructed explicitly from
Theorem~\ref{thm:kolmogorov-extension} or is already part of the
hypotheses of the relevant statement.  Beyond these foundational
facts, the principal probabilistic and ergodic ingredients used in the
main construction are stated and proved in the form required.

\begin{theorem}[Bounded convergence theorem]\label{thm:bounded-convergence}
Let $(\Omega,\mathcal F,\mu)$ be a probability space.  If measurable
functions $f_n:\Omega\to\mathbb R$ satisfy $f_n\to f$ $\mu$-almost
surely and there is a finite constant $C$ such that
$|f_n|\le C$ $\mu$-almost surely for every $n$, then $f$ is integrable
and
\[
  \int f_n\,d\mu\longrightarrow\int f\,d\mu.
\]
\end{theorem}

\begin{proof}
After changing all functions on one null set, assume that the convergence
and bounds hold everywhere.  The pointwise limit $f$ is measurable and
$|f|\le C$.  Put
\[
  g_m:=\sup_{n\ge m}|f_n-f|.
\]
Then $g_m$ is measurable, $0\le g_m\le2C$, and $g_m\downarrow0$.
For every $\epsilon>0$, the events $\{g_m>\epsilon\}$ decrease to the
empty set, so continuity from above gives
$\mu\{g_m>\epsilon\}\to0$.  If $n\ge m$, then
\[
 \int|f_n-f|\,d\mu
 \le \epsilon+2C\,\mu\{g_m>\epsilon\}.
\]
First let $m\to\infty$ and then $\epsilon\downarrow0$.  Thus
$\int|f_n-f|\,d\mu\to0$, which implies the asserted convergence of
integrals.
\end{proof}

\begin{theorem}[Borel--Cantelli lemmas]\label{thm:borel-cantelli}
Let $(E_s)_{s\ge1}$ be measurable events in a probability space
$(\Omega,\mathcal F,\Prob)$.
\begin{enumerate}[label=\textup{(\alph*)}]
\item If $\sum_s\Prob(E_s)<\infty$, then
      $\Prob(E_s\ \mathrm{i.o.})=0$.
\item If the events $E_s$ are mutually independent and
      $\sum_s\Prob(E_s)=\infty$, then
      $\Prob(E_s\ \mathrm{i.o.})=1$.
\end{enumerate}
Here $\{E_s\ \mathrm{i.o.}\}:=\bigcap_{m\ge1}\bigcup_{s\ge m}E_s$.
\end{theorem}

\begin{proof}
For part~(a), the union bound gives
\[
 \Prob\!\left(\bigcup_{s\ge m}E_s\right)
 \le\sum_{s\ge m}\Prob(E_s)\longrightarrow0.
\]
The events on the left decrease with $m$, so continuity from above gives
$\Prob(E_s\ \mathrm{i.o.})=0$.

For part~(b), independence and $1-u\le e^{-u}$ give, for $n\ge m$,
\[
 \Prob\!\left(\bigcap_{s=m}^n E_s^c\right)
 =\prod_{s=m}^n(1-\Prob(E_s))
 \le \exp\!\left(-\sum_{s=m}^n\Prob(E_s)\right)\longrightarrow0.
\]
Continuity from above therefore yields
$\Prob(\bigcap_{s\ge m}E_s^c)=0$ for every $m$.  The event that only
finitely many $E_s$ occur is
$\bigcup_{m\ge1}\bigcap_{s\ge m}E_s^c$, a countable union of null
sets.  Its complement is $\{E_s\ \mathrm{i.o.}\}$.
\end{proof}

\begin{lemma}[Finite-horizon maximal ergodic lemma]
\label{lem:maximal-ergodic}
Let $(\Omega,\mathcal F,\mu,S)$ be probability preserving, let
$h\in L^1(\mu)$, and put
\[
 H_m:=\sum_{i=0}^{m-1}h\circ S^i,
 \qquad
 E_N:=\left\{\max_{1\le m\le N}H_m>0\right\}.
\]
Then $\int_{E_N}h\,d\mu\ge0$.
\end{lemma}

\begin{proof}
Set $H_N^*:=\max(0,H_1,\ldots,H_N)$.  If $x\in E_N$, then every
$H_m(x)$ is at most $h(x)+H_N^*(Sx)$, and consequently
$H_N^*(x)\le h(x)+H_N^*(Sx)$.  Since $H_N^*=0$ on $E_N^c$,
invariance of $\mu$ gives
\begin{align*}
 \int_{E_N}h\,d\mu
 &\ge \int_{E_N}H_N^*\,d\mu
      -\int_{E_N}H_N^*\circ S\,d\mu\\
 &\ge \int H_N^*\,d\mu-\int H_N^*\circ S\,d\mu=0.
\end{align*}
\end{proof}

\begin{lemma}[Maximal average inequality]\label{lem:maximal-average}
Under the hypotheses of Lemma~\ref{lem:maximal-ergodic}, define
$A_mh:=m^{-1}\sum_{i=0}^{m-1}h\circ S^i$.  For every $\epsilon>0$,
\[
 \mu\left\{\sup_{m\ge1}|A_mh|>\epsilon\right\}
 \le \frac{2\|h\|_{L^1(\mu)}}{\epsilon}.
\]
\end{lemma}

\begin{proof}
For fixed $N$, apply Lemma~\ref{lem:maximal-ergodic} to
$h-\epsilon$.  On
$E_N^+:=\{\max_{1\le m\le N}A_mh>\epsilon\}$ this gives
\[
 \epsilon\mu(E_N^+)\le\int_{E_N^+}h\,d\mu\le\|h\|_1.
\]
Applying the same argument to $-h-\epsilon$ gives the identical bound
for $E_N^-:=\{\min_{1\le m\le N}A_mh<-\epsilon\}$.  Letting
$N\to\infty$ and using continuity from below proves the claim by the
union bound.
\end{proof}

\begin{lemma}[Invariant real functions in an ergodic system]
\label{lem:invariant-function}
Let $(\Omega,\mathcal F,\mu,S)$ be probability preserving and ergodic.
Let $g:\Omega\to\mathbb R$ be measurable and finite $\mu$-almost
surely.  If $g\circ S=g$ $\mu$-almost surely, then there is a constant
$c\in\mathbb R$ such that $g=c$ $\mu$-almost surely.
\end{lemma}

\begin{proof}
For each rational $r$, the identity $g\circ S=g$ almost surely
implies
\[
 \mu\bigl(\{g<r\}\mathbin\triangle S^{-1}\{g<r\}\bigr)=0.
\]
Ergodicity therefore gives $\mu\{g<r\}\in\{0,1\}$.  Because $g$ is finite almost surely,
\[
 \mu\{g<r\}\longrightarrow0\quad(r\to-\infty,\ r\in\mathbb Q),
 \qquad
 \mu\{g<r\}\longrightarrow1\quad(r\to+\infty,\ r\in\mathbb Q).
\]
Define
\[
 c:=\inf\{r\in\mathbb Q:\mu(g<r)=1\}.
\]
The preceding limits show that $c$ is finite.  If $r<c$ is rational,
then $\mu(g<r)=0$; if $r>c$ is rational, the definition of the infimum
yields a $q\in\mathbb Q$ with $q<r$ and $\mu(g<q)=1$, and monotonicity
then gives $\mu(g<r)=1$.  Taking rational sequences
$r\uparrow c$ and $r\downarrow c$ yields
$\mu(g<c)=0$ and $\mu(g\le c)=1$.  Thus $g=c$ almost surely.
\end{proof}

\begin{lemma}[Bounded approximation in \(L^1\)]
\label{lem:L1-bounded-approximation}
Let $(\Omega,\mathcal F,\mu)$ be a probability space and
$f\in L^1(\mu)$.  For every $\epsilon>0$ there is a bounded measurable
function $g$ such that $\|f-g\|_{L^1(\mu)}<\epsilon$.
\end{lemma}

\begin{proof}
For $M\ge1$, define the truncation
\[
 f^{(M)}:=\max\{-M,\min\{f,M\}\}.
\]
It is bounded and measurable, and
\[
 |f-f^{(M)}|\le |f|\,\one_{\{|f|>M\}}.
\]
The functions $|f|\one_{\{|f|\le M\}}$ increase pointwise to
$|f|$.  By monotone convergence,
\[
 \int |f|\one_{\{|f|>M\}}\,d\mu\longrightarrow0.
\]
Choosing $M$ sufficiently large proves the claim.
\end{proof}

\begin{theorem}[Birkhoff's pointwise ergodic theorem, ergodic form]
\label{thm:birkhoff}
Let $(\Omega,\mathcal F,\mu,S)$ be a probability-preserving ergodic
system and let $f\in L^1(\mu)$.  Then
\[
  \frac1n\sum_{i=0}^{n-1}f\circ S^i
  \longrightarrow \int f\,d\mu
  \qquad\mu\text{-almost surely}.
\]
\end{theorem}

\begin{proof}
First suppose that $f$ is essentially bounded.  Choose a finite
$M$ and a bounded measurable function $f^{\flat}$ such that
$f^{\flat}=f$ almost surely and $|f^{\flat}|\le M$ everywhere.  If
$N:=\{f^{\flat}\ne f\}$, then measure preservation gives
\[
 \mu\!\left(\bigcup_{i\ge0}S^{-i}N\right)=0.
\]
Outside this null set, $A_nf^{\flat}=A_nf$ for every $n$.  It therefore
suffices to prove the bounded case for $f^{\flat}$.  Replacing $f$ by
$f^{\flat}$, assume from now on that $|f|\le M$ everywhere.  Write
\[
  \overline f:=\limsup_{n\to\infty}A_nf,
  \qquad
  \underline f:=\liminf_{n\to\infty}A_nf.
\]
Both functions are measurable and take values in $[-M,M]$.  Moreover,
for every $x\in\Omega$,
\[
 A_nf(Sx)-A_nf(x)
 =\frac{f(S^n x)-f(x)}{n}\longrightarrow0.
\]
Consequently,
\[
  \overline f\circ S=\overline f,
  \qquad
  \underline f\circ S=\underline f
\]
pointwise; in particular, no choice of exceptional sets is hidden in
the invariance argument.

Fix $\epsilon>0$ and put
$h:=f-\overline f+\epsilon$.  Since $\overline f$ is pointwise
invariant,
\[
  A_nh=A_nf-\overline f+\epsilon,
  \qquad
  \limsup_{n\to\infty}A_nh=\epsilon
\]
at every point.  Hence the events
\[
  E_N(h):=
  \left\{\max_{1\le m\le N}\sum_{i=0}^{m-1}h\circ S^i>0\right\}
\]
increase to all of $\Omega$.  Lemma~\ref{lem:maximal-ergodic} gives
$\int_{E_N(h)}h\,d\mu\ge0$ for every $N$.  Since $h$ is bounded,
Theorem~\ref{thm:bounded-convergence}, applied to
$h\one_{E_N(h)}\to h$, gives
\[
  \int h\,d\mu\ge0,
  \qquad\text{hence}\qquad
  \int\overline f\,d\mu\le\int f\,d\mu+\epsilon.
\]
Applying the same argument to
$h':=\underline f-f+\epsilon$ gives
\[
  \int\underline f\,d\mu\ge\int f\,d\mu-\epsilon.
\]
Letting $\epsilon\downarrow0$ yields
\[
  \int\overline f\,d\mu\le\int f\,d\mu
  \le\int\underline f\,d\mu.
\]
Because $\underline f\le\overline f$ pointwise, the nonnegative bounded
function $\overline f-\underline f$ has integral zero.  Therefore
$\underline f=\overline f$ almost surely, and $A_nf$ converges almost
surely to the finite function $g:=\overline f$.  The pointwise
invariance of $\overline f$ implies $g\circ S=g$ almost surely.
Lemma~\ref{lem:invariant-function} therefore gives $g=c$ almost surely
for some constant $c$.  The two integral inequalities above, or
bounded convergence applied to $A_nf\to c$, give
$c=\int f\,d\mu$.

Now let $f\in L^1(\mu)$.  By
Lemma~\ref{lem:L1-bounded-approximation}, choose bounded measurable
$f_j$ such that $\|f-f_j\|_1\le2^{-3j}$.  Lemma~\ref{lem:maximal-average} gives
\[
 \mu\left\{\sup_{n\ge1}|A_n(f-f_j)|>2^{-j}\right\}
 \le 2^{1-2j}.
\]
The series of these upper bounds converges, so
Theorem~\ref{thm:borel-cantelli}(a) implies that, almost surely, for all
sufficiently large $j$,
$\sup_n|A_n(f-f_j)|\le2^{-j}$.  For every $j$, the bounded case gives
$A_nf_j\to\int f_j\,d\mu$ almost surely.  Intersecting the resulting
countably many conull events, and using
$|\int(f-f_j)\,d\mu|\le2^{-3j}$, gives
\[
 \limsup_{n\to\infty}\left|A_nf-\int f\,d\mu\right|
 \le2^{-j}+2^{-3j}
\]
for all sufficiently large $j$.  Letting $j\to\infty$ completes the
proof.
\end{proof}

\begin{lemma}[Approximation by a generating algebra]
\label{lem:algebra-approximation}
Let $(\Omega,\mathcal F,\mu)$ be a probability space and let
$\mathcal A$ be an algebra of subsets of $\Omega$ such that
$\sigma(\mathcal A)=\mathcal F$.  Then, for every $A\in\mathcal F$ and
$\epsilon>0$, there is $A'\in\mathcal A$ such that
$\mu(A\mathbin\triangle A')<\epsilon$.

In particular, in a finite or countable product probability space,
every measurable event can be approximated in measure by an event in
the cylinder algebra generated by measurable rectangles depending on
only finitely many coordinates.
\end{lemma}

\begin{proof}
Let $\mathcal D$ be the collection of all $A\in\mathcal F$ having the
stated approximation property for every $\epsilon>0$.  The class
$\mathcal D$ contains $\mathcal A$ and is closed under complements.  If
$A_i\in\mathcal D$ and $A=\bigcup_{i\ge1}A_i$, continuity from below
gives an $N$ such that
\[
  \mu\!\left(A\setminus\bigcup_{i=1}^N A_i\right)<\epsilon/2.
\]
Choose $A_i'\in\mathcal A$ with
$\mu(A_i\mathbin\triangle A_i')<\epsilon/(2N)$ for $1\le i\le N$.
Since $\mathcal A$ is an algebra, $A':=\bigcup_{i=1}^N A_i'$ belongs to
$\mathcal A$, and
\[
  \mu(A\mathbin\triangle A')
  \le \mu\!\left(A\setminus\bigcup_{i=1}^N A_i\right)
     +\sum_{i=1}^N\mu(A_i\mathbin\triangle A_i')<\epsilon.
\]
Thus $\mathcal D$ is a sigma-field containing $\mathcal A$, so
$\mathcal D=\mathcal F$.  The product-space assertion follows because
the indicated cylinder algebra generates the product sigma-field.
\end{proof}

\begin{theorem}[Extension for countably many finite coordinates]
\label{thm:kolmogorov-extension}
Let $I$ be finite or countable and, for each $i\in I$, let $E_i$ be a
nonempty finite set with its full sigma-field.  Suppose that for every
finite $J\subset I$ a probability measure $\mu_J$ on
$\prod_{i\in J}E_i$ is given and that these measures are projectively
consistent: whenever $J\subset J'$ are finite, the marginal of
$\mu_{J'}$ on $\prod_{i\in J}E_i$ equals $\mu_J$.  Then there is a
unique probability measure $\mu$ on $\prod_{i\in I}E_i$, with its
product sigma-field, whose marginal on every finite subproduct is
$\mu_J$.
\end{theorem}

\begin{proof}
For each finite $J\subset I$, write
\[
 E_J:=\prod_{i\in J}E_i,
 \qquad
 \pi_J:E:=\prod_{i\in I}E_i\longrightarrow E_J,
 \qquad
 \pi_J(x):=(x_i)_{i\in J}.
\]
If $J\subset K\subset I$ are finite, also write
$\pi_{K,J}:E_K\to E_J$ for the corresponding coordinate projection.
A finite-coordinate cylinder is a set of the form
$\pi_J^{-1}(B)$ with $J\subset I$ finite and $B\subset E_J$.  Because
each $E_i$ has its full sigma-field, every such $B$ is measurable.
The collection $\mathcal A$ of all finite-coordinate cylinders is an
algebra: complements preserve the same coordinate set, and finite
unions can be represented after replacing the coordinate sets by their
finite union.  By definition, the product sigma-field on $E$ is
$\sigma(\mathcal A)$.

If $I$ is finite, take $\mu:=\mu_I$.  Its marginal on every
$E_J$, $J\subset I$, is $\mu_J$ by projective consistency.  Since in
this case $\mathcal A$ is the full sigma-field on the finite set $E$,
these marginals determine $\mu$ uniquely.  Hence it remains to treat
the case in which $I$ is countably infinite.

Fix an enumeration
\[
 I=\{i_1,i_2,\ldots\}.
\]
Recursively along this enumeration, choose one reference point
$e_{i_r}\in E_{i_r}$ for every $r\ge1$.  These reference coordinates
let us extend any point of a finite subproduct $E_J$ to a point of $E$.
Define a set function $\mu_0:\mathcal A\to[0,1]$ by
\[
 \mu_0\bigl(\pi_J^{-1}(B)\bigr):=\mu_J(B).
\]
We first verify that this definition is independent of the chosen
cylinder representation.  Suppose
\[
 \pi_J^{-1}(B)=\pi_{J'}^{-1}(B')
\]
for finite $J,J'\subset I$.  Put $K:=J\cup J'$.  Then
\[
 \pi_{K,J}^{-1}(B)=\pi_{K,J'}^{-1}(B').
\]
Indeed, if a point of $E_K$ belonged to exactly one of these two sets,
extending it outside $K$ by the reference coordinates $(e_i)$ would
produce a point of $E$ belonging to exactly one of the two displayed
cylinders, contrary to their equality.  Projective consistency now
gives
\[
 \mu_J(B)
 =\mu_K\bigl(\pi_{K,J}^{-1}(B)\bigr)
 =\mu_K\bigl(\pi_{K,J'}^{-1}(B')\bigr)
 =\mu_{J'}(B').
\]
Thus $\mu_0$ is well defined.

The set function $\mu_0$ is finitely additive.  To see this, let
$A_1,\ldots,A_m\in\mathcal A$ be pairwise disjoint.  Choose a finite
coordinate set $K\subset I$ on which all the $A_r$ depend, and write
$A_r=\pi_K^{-1}(B_r)$ with $B_r\subset E_K$.  The sets
$B_1,\ldots,B_m$ are pairwise disjoint: otherwise a point in an
intersection $B_r\cap B_s$ could be extended to a point of
$A_r\cap A_s$.  Therefore
\[
 \mu_0\!\left(\bigcup_{r=1}^m A_r\right)
 =\mu_K\!\left(\bigcup_{r=1}^m B_r\right)
 =\sum_{r=1}^m\mu_K(B_r)
 =\sum_{r=1}^m\mu_0(A_r).
\]
In particular, $\mu_0(\varnothing)=0$.  Since $I$ is countably
infinite in the case now under consideration,
$E=\pi_{\{i_1\}}^{-1}(E_{i_1})$, and hence
\[
 \mu_0(E)=\mu_{\{i_1\}}(E_{i_1})=1.
\]

We next prove continuity of $\mu_0$ at the empty set.  In fact, we prove
the stronger cylinder-intersection statement
\[
 \text{if }A_1\supset A_2\supset\cdots,\quad
 A_n\in\mathcal A,\quad\text{and every }A_n\ne\varnothing,
 \quad\text{then }\bigcap_{n\ge1}A_n\ne\varnothing.
 \tag{*}
\]
Choose $x^{(n)}\in A_n$ for every $n$.  Since $E_{i_1}$ is finite,
there is an infinite set $N_1\subset\mathbb N$ on which the coordinate
$x^{(n)}_{i_1}$ is constant.  Inductively, after choosing an infinite
$N_{r-1}$, finiteness of $E_{i_r}$ gives an infinite
$N_r\subset N_{r-1}$ on which $x^{(n)}_{i_r}$ is constant.  Choose
strictly increasing integers $n_r$ with $n_r\in N_r$.  For every fixed
$q$, the coordinates $x^{(n_r)}_{i_q}$ are then constant for all
$r\ge q$.  Define $x\in E$ by letting $x_{i_q}$ be this eventual
constant value.

Fix $m\ge1$.  Because $A_m$ is a cylinder, there is a finite set
$J_m\subset I$ such that membership in $A_m$ depends only on the
coordinates in $J_m$.  Choose $r$ so large that
\[
 n_r\ge m
 \quad\text{and}\quad
 J_m\subset\{i_1,\ldots,i_r\}.
\]
The sequence $(A_n)$ is decreasing, so
$x^{(n_r)}\in A_{n_r}\subset A_m$.  Moreover,
$x^{(n_r)}$ and $x$ agree on every coordinate in $J_m$.  Hence
$x\in A_m$.  Since $m$ was arbitrary, $x\in\bigcap_{m\ge1}A_m$,
which proves~\textup{(*)}.

Consequently, if $A_n\downarrow\varnothing$ with $A_n\in\mathcal A$,
then some $A_N$ must be empty; otherwise~\textup{(*)} would give a
point in their intersection.  Thus $\mu_0(A_n)=0$ for all $n\ge N$,
and in particular
\[
 \mu_0(A_n)\longrightarrow0.
\]
This is continuity of $\mu_0$ at the empty set.

We now verify countable additivity in precisely the form required of a
premeasure.  Let $A,A_1,A_2,\ldots\in\mathcal A$, suppose the $A_n$
are pairwise disjoint, and suppose
\[
 A=\bigsqcup_{n\ge1}A_n.
\]
For $N\ge1$, set
\[
 R_N:=A\setminus\bigcup_{n=1}^N A_n.
\]
Because $\mathcal A$ is an algebra, $R_N\in\mathcal A$; moreover
$R_N\downarrow\varnothing$.  Finite additivity gives
\[
 \mu_0(A)=\sum_{n=1}^N\mu_0(A_n)+\mu_0(R_N).
\]
Letting $N\to\infty$ and using continuity at the empty set yields
\[
 \mu_0(A)=\sum_{n\ge1}\mu_0(A_n).
\]
Hence $\mu_0$ is a finite premeasure on $\mathcal A$.

By the Carath\'eodory extension theorem, $\mu_0$ extends to a measure
$\mu$ on $\sigma(\mathcal A)$, the product sigma-field of $E$.
Because $E\in\mathcal A$ and $\mu_0(E)=1$, the extension is a
probability measure.  For every finite $J\subset I$ and every
$B\subset E_J$,
\[
 \mu\bigl(\pi_J^{-1}(B)\bigr)
 =\mu_0\bigl(\pi_J^{-1}(B)\bigr)
 =\mu_J(B),
\]
so the marginal of $\mu$ on $E_J$ is $\mu_J$.

Finally, if $\widetilde\mu$ is another probability measure on the
product sigma-field with the same finite-dimensional marginals, then
$\widetilde\mu$ and $\mu$ agree on every member of $\mathcal A$.
Since $\mathcal A$ is an algebra generating the product sigma-field and
both measures are finite, the uniqueness clause of the
Carath\'eodory extension theorem implies
$\widetilde\mu=\mu$.  This proves both existence and uniqueness.
\end{proof}

\begin{lemma}[Stationary bilateral finite-state Markov law]
\label{lem:bilateral-markov}
Let $E$ be a finite set, let $M=(M(a,b))_{a,b\in E}$ be a stochastic
matrix, and let $\nu$ be a probability row vector satisfying
$\nu M=\nu$.  There exists a unique probability law
$\mathsf M_{\nu,M}$ on $E^{\Z}$, equipped with its product sigma-field,
which is invariant under the bilateral left shift
$\sigma((z_t)_{t\in\Z})=(z_{t+1})_{t\in\Z}$ and under which the coordinate process
$(Z_t)_{t\in\Z}$ is a stationary Markov chain with one-time marginal $\nu$ and
transition matrix $M$.  Its finite-dimensional distributions are
characterized by
\[
 \mathsf M_{\nu,M}\{Z_{t_0}=z_0,\ldots,Z_{t_m}=z_m\}
 =\nu(z_0)\prod_{i=1}^m M^{t_i-t_{i-1}}(z_{i-1},z_i)
\]
for integers $t_0<\cdots<t_m$ and states $z_0,\ldots,z_m\in E$.
If
\[
  M^n(a,b)\longrightarrow\nu(b)
  \qquad(n\to\infty)
\]
for all $a,b\in E$, then the bilateral shift on
$(E^{\Z},\mathsf M_{\nu,M})$ is mixing.
\end{lemma}

\begin{proof}
We first verify projective consistency.  Summing over a terminal
state uses $\sum_bM^r(a,b)=1$.  Summing over an initial state uses
$\sum_a\nu(a)M^r(a,b)=(\nu M^r)(b)=\nu(b)$.  Summing over an
intermediate state uses the Chapman--Kolmogorov identity
\[
 \sum_cM^r(a,c)M^s(c,b)=M^{r+s}(a,b).
\]
Repeated marginalization therefore gives every required lower-dimensional
distribution.  Theorem~\ref{thm:kolmogorov-extension} gives a unique law
on $E^{\Z}$.  The displayed formula is unchanged by translating all
times by the same integer, so the law is shift invariant.  The same
formula, or equivalently division by any positive-probability finite
past cylinder, gives the Markov property with transition matrix $M$ and
one-time marginal $\nu$.

Let $A$ and $B$ be cylinder events.  Choose an integer $u$ not smaller
than every time coordinate on which $A$ depends, and an integer $v$ not
larger than every time coordinate on which $B$ depends; no ordering of
$u$ and $v$ is required.  Define, using only ratios of probabilities of
finite-cylinder events,
\[
  h_B(b):=
  \begin{cases}
    \displaystyle
    \frac{\mathsf M_{\nu,M}(B\cap\{Z_v=b\})}{\nu(b)},
       &\nu(b)>0,\\[2ex]
    0,&\nu(b)=0.
  \end{cases}
\]
Thus no regular conditional-probability existence theorem is being
invoked.  By finite additivity,
\[
  \sum_{b\in E}\nu(b)h_B(b)=\mathsf M_{\nu,M}(B).
\]

For completeness, we record the finite-cylinder factorization used
below.  If $C$ is a cylinder event depending only on coordinates at
times at most $u$, and $D$ is a cylinder event depending only on
coordinates at times at least $w\ge u$, then
\begin{equation}\label{eq:markov-cylinder-factorization}
 \mathsf M_{\nu,M}(C\cap D)
 =\sum_{a,b\in E}
   \mathsf M_{\nu,M}(C\cap\{Z_u=a\})
   M^{w-u}(a,b)
   \frac{\mathsf M_{\nu,M}(D\cap\{Z_w=b\})}
        {\mathsf M_{\nu,M}\{Z_w=b\}},
\end{equation}
where the quotient is defined to be $0$ when its denominator is $0$.
To verify \eqref{eq:markov-cylinder-factorization}, first take $C$ and
$D$ to be cylinder atoms.  The defining finite-dimensional formula
for $\mathsf M_{\nu,M}$ splits the product of transition factors at
times $u$ and $w$; summing over any unmentioned intermediate states
uses the Chapman--Kolmogorov identity.  Finite disjoint unions then
give the formula for arbitrary cylinder events $C,D$.

We also check that the zero-denominator convention causes no missing
terms.  If $\nu(b)=0$, then for every $r\ge1$ stationarity gives
\[
 0=\nu(b)=\sum_a\nu(a)M^r(a,b),
\]
so nonnegativity implies $M^r(a,b)=0$ whenever $\nu(a)>0$.  If $r=0$,
the only potentially nonzero term has $a=b$, but then
$\mathsf M_{\nu,M}(C\cap\{Z_u=b\})\le\nu(b)=0$.

Now let $n\ge\max\{0,u-v\}$ and put
\[
  w:=n+v,\qquad r:=w-u=n+v-u\in\N_0.
\]
The event $\sigma^{-n}B$ depends only on coordinates at times at least
$w$.  Stationarity gives
\[
 \frac{\mathsf M_{\nu,M}
   (\sigma^{-n}B\cap\{Z_w=b\})}
      {\mathsf M_{\nu,M}\{Z_w=b\}}
 =h_B(b)
\]
with the same zero-denominator convention.  Applying
\eqref{eq:markov-cylinder-factorization} with
$C=A$ and $D=\sigma^{-n}B$ therefore yields
\[
 \mathsf M_{\nu,M}(A\cap\sigma^{-n}B)
 =\sum_{a,b\in E}
   \mathsf M_{\nu,M}(A\cap\{Z_u=a\})
   M^{n+v-u}(a,b)h_B(b).
\]
All sums are finite, and the exponent is a nonnegative integer.
Taking $n\to\infty$ and using the assumed transition-probability limit
gives
\begin{align*}
 \mathsf M_{\nu,M}(A\cap\sigma^{-n}B)
 &\longrightarrow
 \sum_{a,b\in E}
   \mathsf M_{\nu,M}(A\cap\{Z_u=a\})\nu(b)h_B(b)\\
 &=\mathsf M_{\nu,M}(A)\mathsf M_{\nu,M}(B),
\end{align*}
because $\sum_b\nu(b)h_B(b)=\mathsf M_{\nu,M}(B)$.
To extend the limit, fix measurable $A,B$ and $\epsilon>0$ and choose
cylinder-algebra events $A',B'$ with
$\mathsf M_{\nu,M}(A\mathbin\triangle A')<\epsilon$ and
$\mathsf M_{\nu,M}(B\mathbin\triangle B')<\epsilon$.  Invariance gives,
uniformly in $n$,
\[
 \left|\mathsf M_{\nu,M}(A\cap\sigma^{-n}B)
       -\mathsf M_{\nu,M}(A'\cap\sigma^{-n}B')\right|\le2\epsilon,
\]
and
\[
 \left|\mathsf M_{\nu,M}(A)\mathsf M_{\nu,M}(B)
       -\mathsf M_{\nu,M}(A')\mathsf M_{\nu,M}(B')\right|
 \le2\epsilon.
\]
Apply the cylinder limit and then let $\epsilon\downarrow0$.
\end{proof}

\section{Preliminary mixing facts}\label{sec:mixing-facts}

We next prove the elementary mixing facts needed for the construction.

\begin{lemma}[Products and factors of mixing systems]\label{lem:mixing-products}
Let $(\Omega_r,\mathcal F_r,\mu_r,S_r)$, $r\in R$, be
probability-preserving mixing systems, where $R$ is finite or countable.
Assume that the product probability measure
$\mu:=\bigotimes_{r\in R}\mu_r$ on
$\bigotimes_{r\in R}\mathcal F_r$ is given (for finite $R$ this is the
ordinary finite product measure).
\begin{enumerate}[label=\textup{(\alph*)}]
\item The coordinatewise product transformation on
      $(\prod_{r\in R}\Omega_r,\bigotimes_{r\in R}\mathcal F_r,
      \bigotimes_{r\in R}\mu_r)$ is mixing.
\item Every measurable factor of a mixing system is mixing.
\item Every mixing system is ergodic.
\item Every measurable factor of an ergodic probability-preserving
      system is ergodic.
\end{enumerate}
\end{lemma}

\begin{proof}
For a finite product and measurable rectangles
$A=\prod_r A_r$, $B=\prod_r B_r$, mixing follows from
\[
  \mu(A\cap S^{-m}B)
  =\prod_r\mu_r(A_r\cap S_r^{-m}B_r)
  \longrightarrow
  \prod_r\mu_r(A_r)\mu_r(B_r)=\mu(A)\mu(B).
\]
By finite disjointification and finite additivity, the same limit holds
for $A,B$ in the algebra generated by measurable rectangles.
Lemma~\ref{lem:algebra-approximation} then extends the conclusion to all
measurable $A,B$.

For a countable product, first take $A,B$ in the finite-coordinate
cylinder algebra.  Both depend on a common finite set of product
coordinates and hence may be viewed as events in a finite product, for
which the desired limit has just been proved.  For arbitrary measurable
$A,B$ and $\epsilon>0$, Lemma~\ref{lem:algebra-approximation} gives
cylinder-algebra events $A',B'$ with
$\mu(A\mathbin\triangle A')<\epsilon$ and
$\mu(B\mathbin\triangle B')<\epsilon$.  Invariance gives, uniformly in
$m$,
\[
 \bigl|\mu(A\cap S^{-m}B)-\mu(A'\cap S^{-m}B')\bigr|\le 2\epsilon
\]
and
\[
 \bigl|\mu(A)\mu(B)-\mu(A')\mu(B')\bigr|\le 2\epsilon.
\]
Apply mixing to $A',B'$ and then let $\epsilon\downarrow0$.

For the mixing-factor assertion, let $\Phi$ be a measurable map
satisfying $\Phi\circ S=R\circ\Phi$ almost surely and let
$\nu=\Phi_\#\mu$.  After removing a countable union of null sets, the
intertwining identity holds simultaneously for all nonnegative
iterates.  Hence, for measurable factor events $C,D$,
\[
 \nu(C\cap R^{-m}D)
 =\mu(\Phi^{-1}C\cap S^{-m}\Phi^{-1}D)
 \longrightarrow\nu(C)\nu(D).
\]

If $A$ is $S$-invariant modulo null sets in a mixing system, then
$\mu(A)=\mu(A\cap S^{-m}A)\to\mu(A)^2$, so
$\mu(A)\in\{0,1\}$.  Thus mixing implies ergodicity.  Finally, if the
source system is ergodic and $C$ is invariant modulo null sets under a
factor transformation $R$, then $\Phi^{-1}C$ is invariant modulo null
sets under $S$.  Therefore
$\nu(C)=\mu(\Phi^{-1}C)\in\{0,1\}$, proving factor ergodicity.
\end{proof}

\begin{lemma}[Existence and mixing of the geometric marker shift]
\label{lem:bernoulli-mixing}
There is a unique i.i.d.\ law $\pi^{\Z}$ on $\N^{\Z}$ with
$\pi_k:=\pi\{k\}=2^{-k}$ for $k\ge1$, and the bilateral left shift on
$(\N^{\Z},\pi^{\Z})$ is mixing.
\end{lemma}

\begin{proof}
Apply Theorem~\ref{thm:kolmogorov-extension} to the finite coordinate
spaces $\{0,1\}$ indexed by $\Z\times\N$, with all coordinates fair
and independent, and call the resulting law $\beta$.  For $b=(b_{t,r})$, define
\[
 K_t(b):=\min\{r\ge1:b_{t,r}=1\}
\]
on the event that this set is nonempty, and define $K_t=1$ otherwise.
This is measurable because, for $k\ge2$,
\[
 \{K_t=k\}=\{b_{t,1}=\cdots=b_{t,k-1}=0,\ b_{t,k}=1\},
\]
while $\{K_t=1\}$ is the union of $\{b_{t,1}=1\}$ and the measurable
all-zero-row event.
For fixed $t$, the exceptional all-zero-row event has probability
$\lim_{m\to\infty}2^{-m}=0$; the union over $t\in\Z$ is still null.
Thus the convention $K_t=1$ on that event changes no one-row
probability, and
\[
 \beta\{K_t=k\}=2^{-k}\qquad(k\ge1).
\]
Moreover, each $K_t$ is measurable with respect to row $t$, and distinct
rows are independent.  Therefore, for any distinct
$t_1,\ldots,t_m$ and $k_1,\ldots,k_m\ge1$, the equality
\[
 \beta\{K_{t_1}=k_1,\ldots,K_{t_m}=k_m\}
 =\prod_{j=1}^m 2^{-k_j}
\]
holds exactly, including when some $k_j=1$; the added all-zero-row
pieces are null.
The pushforward law is therefore the desired i.i.d.\ law $\pi^{\Z}$.
Any other probability law with the same i.i.d.\ finite-dimensional
distributions agrees with it on the finite-coordinate cylinder algebra;
the uniqueness clause of the Carath\'eodory extension theorem therefore
gives equality on the product sigma-field.

If cylinder events $A$ and $B$ depend on finite sets of time
coordinates, then $A$ and $\sigma^{-m}B$ depend on disjoint coordinate
sets for all sufficiently large $m$ and hence are independent.
Lemma~\ref{lem:algebra-approximation} extends the mixing limit from
cylinders to all measurable events.
\end{proof}

\section{The synchronizing coded processes}\label{sec:code}

\subsection{Environments and markers}

Put
\[
  \Omega_Y:=\bigl(\{0,1\}^{\N}\bigr)^{\Z}
\]
with its product Borel sigma-field.  An element $y\in\Omega_Y$ is
written
\[
  y=(y_t)_{t\in\Z},
  \qquad y_t=(y_t^1,y_t^2,\ldots).
\]
Let $\sigma$ denote the bilateral left shift:
$(\sigma y)_t=y_{t+1}$.  A probability measure $\eta$ on $\Omega_Y$
is called an \emph{environment law} if it is $\sigma$-invariant and
mixing, that is,
\[
  \eta(A\cap\sigma^{-m}B)\longrightarrow\eta(A)\eta(B)
  \qquad(m\to\infty)
\]
for all measurable $A,B\subset\Omega_Y$.  By Lemma~\ref{lem:mixing-products}(c), every environment law is
also ergodic.

Let
\[
  \Omega_K:=\N^{\Z},
  \qquad \pi_k:=2^{-k}\quad(k\in\N),
\]
and let $\pi^{\Z}$ be the i.i.d.\ law supplied by
Lemma~\ref{lem:bernoulli-mixing}.  For $J\subset\Z$, write
$\pi^J$ for its marginal on $\N^J$.  Thus $K=(K_t)_{t\in\Z}$ is
i.i.d.\ and $\Prob(K_t=k)=\pi_k$.  The environment and the marker process will
always be independent.

For an environment law $\eta$, define its coordinate means by
\[
  q_k(\eta):=\eta\{y:y_0^k=1\},\qquad k\in\N.
\]

\subsection{A synchronizing binary code}

For $k\in\N$ and $b\in\{0,1\}$, set
\[
  d_0:=01,\qquad d_1:=11,
\]
and define
\[
  w(k,b):=00\,1^k\,01\,d_b.
\]
Its length is
\[
  \ell(k):=k+6.
\]
The word $00$ occurs in $w(k,b)$ only in its first two positions, and
every codeword ends in $1$.  Therefore, in any finite or infinite
concatenation of codewords, the occurrences of $00$ are exactly the
codeword starts.  At such a start, the number $k\ge1$ of consecutive
ones following the initial $00$ is determined by the position of the
next zero; the following bits must be $01$ and then either $01$ or
$11$, which uniquely determines $b$.  Hence the code is synchronizing
and uniquely decodable from every occurrence of $00$.

The mean codeword length is finite and equals
\[
  L:=\sum_{k\ge1}\pi_k\ell(k)
    =\sum_{k\ge1}2^{-k}(k+6)=2+6=8.
\]

\subsection{The stationary suspension and output law}

Fix an environment law $\eta$ and put
\[
  \mu_\eta:=\eta\otimes\pi^{\Z}
\]
on $\Omega_Y\times\Omega_K$.  Define the roof
\[
  r(y,k):=\ell(k_0)
\]
and the suspension space
\[
  \widehat\Omega
  :=\{(y,k,j):y\in\Omega_Y,\ k\in\Omega_K,
                  \ 0\le j<\ell(k_0)\}.
\]
Equip $\widehat\Omega$ with the trace sigma-field inherited from
$\Omega_Y\times\Omega_K\times\N_0$ and with the probability measure
\begin{equation}\label{eq:tower-measure}
  \widehat\mu_\eta(A)
  :=\frac1L
  \int_{\Omega_Y\times\Omega_K}
  \sum_{j=0}^{\ell(k_0)-1}\one_A(y,k,j)
  \,d\eta(y)\,d\pi^{\Z}(k).
\end{equation}
The normalization is correct because
$\int\ell(k_0)\,d\pi^{\Z}=L$.

Define the tower map
\[
  S(y,k,j)
  :=
  \begin{cases}
    (y,k,j+1),&j+1<\ell(k_0),\\
    (\sigma y,\sigma k,0),&j+1=\ell(k_0).
  \end{cases}
\]
It is a measurable bijection, with inverse
\[
  S^{-1}(y,k,j)
  :=
  \begin{cases}
    (y,k,j-1),&j\ge1,\\
    (\sigma^{-1}y,\sigma^{-1}k,\ell(k_{-1})-1),&j=0.
  \end{cases}
\]
If
$w(k_0,y_0^{k_0})=(c_1,\ldots,c_{\ell(k_0)})$, define the symbol map
\[
  F(y,k,j):=c_{j+1}.
\]
The map $F$ is measurable because its level sets are countable unions
of events specified by the measurable coordinates $k_0$, $j$, and
$y_0^k$.  Each map $F\circ S^{i-1}$ is therefore measurable, and the
product sigma-field on $\overline{\cX}$ makes the following bilateral
coding map measurable:
\[
  \Phi:\widehat\Omega\longrightarrow\overline{\cX},
  \qquad
  \Phi(y,k,j)_i:=F(S^{i-1}(y,k,j)),\quad i\in\Z.
\]
Thus the bilateral output process is $X_i:=\Phi_i=F\circ S^{i-1}$.
By construction,
\begin{equation}\label{eq:factor-intertwining}
  \Phi\circ S=\overline T\circ\Phi
\end{equation}
pointwise.  Let
\[
  \overline P_\eta:=\Phi_\#\widehat\mu_\eta.
\]
Let $\operatorname{res}_+:\overline{\cX}\to\cX$ be the restriction
map
\[
 \operatorname{res}_+((x_i)_{i\in\Z})=(x_1,x_2,\ldots).
\]
It is measurable and satisfies
$\operatorname{res}_+\circ\overline T
 =T\circ\operatorname{res}_+$.  Put
\[
  \Phi^+:=\operatorname{res}_+\circ\Phi,
  \qquad
  P_\eta:=(\Phi^+)_\#\widehat\mu_\eta
         =(\operatorname{res}_+)_\#\overline P_\eta.
\]
Then $\Phi^+\circ S=T\circ\Phi^+$ pointwise.

The lifted law factors in a useful way.  Define
\[
  \widehat\Omega_K
  :=\{(k,j):k\in\Omega_K,\ 0\le j<\ell(k_0)\},
\]
equip it with the trace sigma-field inherited from
$\Omega_K\times\N_0$, and define the probability measure
\begin{equation}\label{eq:lambda}
  \lambda(C)
  :=\frac1L\int_{\Omega_K}
     \sum_{j=0}^{\ell(k_0)-1}\one_C(k,j)\,d\pi^{\Z}(k).
\end{equation}
Under the natural bijection
\[
  (y,k,j)\longleftrightarrow (y,(k,j)),
\]
the trace sigma-field on $\widehat\Omega$ is exactly the product of the
sigma-field on $\Omega_Y$ and the trace sigma-field on
$\widehat\Omega_K$.  Indeed, the trace of the ambient product
sigma-field is generated by sets of the form
\[
  B\times\bigl(C\cap\widehat\Omega_K\bigr),
\]
where $B$ is measurable in $\Omega_Y$ and
$C$ is measurable in $\Omega_K\times\N_0$; these are precisely the
rectangular generators of the stated product sigma-field.  Hence, under
this measurable identification,
\begin{equation}\label{eq:factor-lift}
  \widehat\mu_\eta=\eta\otimes\lambda.
\end{equation}
The two sides agree on every measurable rectangle
$B\times C\subset\Omega_Y\times\widehat\Omega_K$ by
\eqref{eq:tower-measure} and \eqref{eq:lambda}; uniqueness of the finite
product measure gives equality on the whole product sigma-field.  Thus
the environment $y$ is independent of the marked suspension
state $(k,j)$, and the law of $(k,j)$ is common to all environment
laws.

\begin{lemma}[Stationarity and ergodicity]\label{lem:ergodic}
For every environment law $\eta$, the measure $\widehat\mu_\eta$ is
$S$-invariant and ergodic.  Consequently, $\overline P_\eta$ is a
stationary ergodic bilateral binary process law and
$P_\eta\in\cE$.
\end{lemma}

\begin{proof}
Let $f$ be bounded and measurable on $\widehat\Omega$.  By
\eqref{eq:tower-measure}, the definition of $S$, and invariance of
$\eta\otimes\pi^{\Z}$ under $\sigma\times\sigma$,
\begin{align*}
 L\int f\circ S\,d\widehat\mu_\eta
 &=\int\left[
   \sum_{j=0}^{\ell(k_0)-2}f(y,k,j+1)
   +f(\sigma y,\sigma k,0)
 \right]d\eta(y)d\pi^{\Z}(k)\\
 &=\int\left[
   \sum_{j=1}^{\ell(k_0)-1}f(y,k,j)+f(y,k,0)
 \right]d\eta(y)d\pi^{\Z}(k)\\
 &=L\int f\,d\widehat\mu_\eta.
\end{align*}
Hence $\widehat\mu_\eta$ is $S$-invariant.

By Lemmas~\ref{lem:mixing-products} and~\ref{lem:bernoulli-mixing}, the base map
$\sigma\times\sigma$ on
$(\Omega_Y\times\Omega_K,\eta\otimes\pi^{\Z})$ is mixing and therefore
ergodic.

Let $A\subset\widehat\Omega$ satisfy
$\one_A\circ S=\one_A$ almost surely, and let
\[
 N:=\{z\in\widehat\Omega:\one_A(Sz)\ne\one_A(z)\}.
\]
Then $\widehat\mu_\eta(N)=0$.  Since $S$ is invertible and
measure preserving,
\[
 N_*:=\bigcup_{m\in\Z}S^mN
\]
is an $S$-invariant null set.  On the $S$-invariant conull set
$G:=\widehat\Omega\setminus N_*$, the equality
$\one_A(Sz)=\one_A(z)$ holds at every point of every full
$S$-orbit.  Hence membership in $A$ is exactly constant along each
$S$-orbit contained in $G$.

For a measurable set $C\subset\widehat\Omega$, define its level-zero
section by
\[
 C^{(0)}:=\{(y,k):(y,k,0)\in C\}.
\]
If $\widehat\mu_\eta(C)=0$, then
\[
 0=L\widehat\mu_\eta(C)
 =\int\sum_{j=0}^{\ell(k_0)-1}\one_C(y,k,j)
       \,d\eta(y)d\pi^{\Z}(k),
\]
and nonnegativity implies
$(\eta\otimes\pi^{\Z})(C^{(0)})=0$.  In particular,
$G^{(0)}$ is conull in the base.  Define
\[
  A_0:=A^{(0)}
  =\{(y,k):(y,k,0)\in A\}.
\]
The first return map from level zero to level zero is
$\sigma\times\sigma$, because
\[
 S^{\ell(k_0)}(y,k,0)=(\sigma y,\sigma k,0).
\]
For every $(y,k)$ such that $(y,k,0)\in G$, orbitwise constancy gives
\[
 \one_{A_0}(\sigma y,\sigma k)
 =\one_A(\sigma y,\sigma k,0)
 =\one_A(y,k,0)
 =\one_{A_0}(y,k).
\]
Thus $A_0$ is invariant modulo an
$\eta\otimes\pi^{\Z}$-null set.  Base ergodicity implies
$(\eta\otimes\pi^{\Z})(A_0)\in\{0,1\}$.

For every allowed level $0\le j<\ell(k_0)$,
\[
 S^j(y,k,0)=(y,k,j).
\]
Therefore, whenever $(y,k,0)\in G$, orbitwise constancy gives
\[
 \one_A(y,k,j)=\one_A(y,k,0)=\one_{A_0}(y,k)
 \qquad(0\le j<\ell(k_0)).
\]
The base set on which $(y,k,0)\notin G$ is null, and integrating its
finite roof height by~\eqref{eq:tower-measure} shows that the preceding
identity holds for $\widehat\mu_\eta$-almost every $(y,k,j)$.  Therefore
\[
  \widehat\mu_\eta(A)
  =\frac1L\int\ell(k_0)\one_{A_0}(y,k)
       \,d\eta(y)d\pi^{\Z}(k),
\]
which is $0$ if $A_0$ is null and $1$ if $A_0$ is conull.  Hence the
tower is ergodic.

Equation~\eqref{eq:factor-intertwining} shows that
$(\overline{\cX},\overline P_\eta,\overline T)$ is a measurable factor
of the tower.  It is probability preserving by the pushforward and
intertwining identities, and it is ergodic by
Lemma~\ref{lem:mixing-products}(d).  Since
$\operatorname{res}_+\circ\overline T
 =T\circ\operatorname{res}_+$, the one-sided system
$(\cX,P_\eta,T)$ is a further measurable factor; it is likewise
probability preserving and ergodic.  Thus $P_\eta\in\cE$.
\end{proof}

\subsection{Coordinate means are intrinsic to the output law}

For a finite binary word $u$, let $[u]\subset\cX$ denote the cylinder
of one-sided sequences beginning with $u$.  Since an output prefix
$00$ can occur only at a codeword start, the event that the output
begins with the complete word $w(k,b)$ is exactly
\[
  \{j=0,\ K_0=k,\ y_0^k=b\}.
\]
Indeed, after the initial $00$, the first following zero fixes $k$,
and the final two-bit block fixes $b$.  Hence
\begin{equation}\label{eq:cylinder}
  P_\eta([w(k,b)])
  =\frac{\pi_k}{L}\,\eta\{y:y_0^k=b\}.
\end{equation}
The cylinders $[w(k,b)]$, $k\in\N$, $b\in\{0,1\}$, are pairwise
disjoint.  Their union is not all of $[00]$ as a subset of the full
sequence space $\cX$, but it exhausts $[00]$ under every coded law:
indeed,
\[
  (\Phi^+)^{-1}([00])=\{(y,k,j):j=0\}
  =\bigsqcup_{k\ge1}\bigsqcup_{b\in\{0,1\}}
    (\Phi^+)^{-1}([w(k,b)]).
\]
Consequently,
\begin{equation}\label{eq:start-intensity}
  P_\eta([00])
  =\widehat\mu_\eta\{j=0\}
  =\frac1L
  =\sum_{k\ge1}\sum_{b\in\{0,1\}}P_\eta([w(k,b)]).
\end{equation}
Moreover,
\begin{equation}\label{eq:q-intrinsic}
  q_k(P_\eta)
  :=\frac{P_\eta([w(k,1)])}
       {P_\eta([w(k,0)])+P_\eta([w(k,1)])}
  =q_k(\eta).
\end{equation}
The denominator is $\pi_k/L>0$.  Therefore every $q_k(P_\eta)$ is
determined by the output law alone, even if the same output law has
more than one environmental representation.

\section{The fixed hypotheses}\label{sec:family}

Let
\[
  \cP:=\{P_\eta:\eta\text{ is an environment law}\}.
\]
Define
\begin{align*}
  H_0&:=\{P\in\cP:q_k(P)\to0\},\\
  H_1&:=\{P\in\cP:q_k(P)\to1\}.
\end{align*}
Lemma~\ref{lem:ergodic} gives $H_0,H_1\subset\cE$.
Equation~\eqref{eq:q-intrinsic} shows that both definitions depend only
on the output law.  The two families are disjoint because a real
sequence cannot converge to both $0$ and $1$.  They are both nonempty.
Indeed, let $y^{(0)}$ be the deterministic environment with
$y_t^k=0$ for all $(t,k)\in\Z\times\N$, and let $y^{(1)}$ be the
deterministic environment with $y_t^k=1$ for all $(t,k)$.  The Dirac laws $\delta_{y^{(0)}}$ and $\delta_{y^{(1)}}$ are
shift-invariant and mixing: each supporting point is fixed by the shift,
and for all measurable $A,B$ the two-set mixing identity holds exactly
for every lag.  Hence they are environment laws.  Their coded output laws satisfy
$q_k(P_{\delta_{y^{(0)}}})=0$ and
$q_k(P_{\delta_{y^{(1)}}})=1$ for every $k$, so they belong to $H_0$
and $H_1$, respectively.

\section{A weakly asymptotically consistent test}\label{sec:weak}

For $n\in\N$, set
\[
  k_n:=\max\left\{1,
  \left\lfloor\frac12\log_2 n\right\rfloor\right\}.
\]
We now define a total deterministic test on every input
$x_1^n\in\{0,1\}^n$.

\begin{enumerate}[label=\textup{\arabic*.},leftmargin=2.2em]
\item If there is no index $r\in\{1,\ldots,n-1\}$ with
      $x_rx_{r+1}=00$, output $0$.
\item Otherwise let $r$ be the first such index.  Starting at $r$,
      attempt to decode consecutive codewords.  At a proposed start,
      a complete valid word must have the form
      $00\,1^k\,01\,d_b$ for a unique $k\ge1$ and
      $b\in\{0,1\}$ and must lie entirely inside $x_1^n$.
      Record every complete valid word and move to the position
      immediately following it.  Stop at the first invalid or
      incomplete attempted word.
\item If no complete valid word was recorded, output $0$.  Otherwise
      discard the first recorded word.  Among the remaining recorded
      words, choose the first whose marker is $k_n$ and output its bit
      $b$.  If there is no such word, output $0$.
\end{enumerate}
This defines a measurable function
\[
  \varphi_n:\{0,1\}^n\to\{0,1\}.
\]
On a genuine coded output, every decoded word is a true codeword: by
synchronization, the first visible $00$ is a true codeword start and
all subsequent complete codewords decode uniquely.

\begin{lemma}[Marker regeneration after the discarded word]
\label{lem:marker-regeneration}
Under $\lambda$,
\begin{equation}\label{eq:length-biased}
  \lambda\{K_0=k\}=\frac{\pi_k\ell(k)}L,
\end{equation}
and the family $(K_t)_{t\ne0}$ is i.i.d.\ with common law $\pi$ and is
independent of the pair $(K_0,j)$.

For a genuine coded output, after the first complete parsed word is
discarded, the successive candidate markers form an i.i.d.\ sequence
with law $\pi$, independent of the environment $y$.
\end{lemma}

\begin{proof}
For $0\le u<\ell(k)$ and for every event $B$ depending only on
$(K_t)_{t\ne0}$, definition~\eqref{eq:lambda} gives
\[
  \lambda\{K_0=k,j=u,(K_t)_{t\ne0}\in B\}
  =\frac{\pi_k}{L}\,\pi^{\Z\setminus\{0\}}(B).
\]
Summing over $u$ proves \eqref{eq:length-biased} and the stated
independence.

If the initial tower level is $j=0$, the first complete visible
codeword has base index $0$ and, after discarding it, the candidate
indices are $1,2,\ldots$.  If $j>0$, the first complete visible
codeword has base index $1$ and, after discarding it, the candidate
indices are $2,3,\ldots$.  Thus the candidate sequence is either
$(K_1,K_2,\ldots)$ or $(K_2,K_3,\ldots)$, with the choice determined by
$(K_0,j)$, which is independent of the future i.i.d.\ marker sequence.
In either case it is i.i.d.\ with law $\pi$.  Independence from the
environment follows from \eqref{eq:factor-lift}.
\end{proof}

For a genuine output prefix $X_1^n$, let $D_n$ denote the number of
complete parsed words remaining after the first complete parsed word
is discarded; set $D_n=0$ when fewer than two complete words are
parsed.

\begin{lemma}[Linear number of usable words]\label{lem:count}
For every $P_\eta\in\cP$,
\[
  \frac{D_n}{n}\longrightarrow\frac1L
  \qquad\widehat\mu_\eta\text{-almost surely}.
\]
The variable $D_n$ is a measurable function of $X_1^n$ by the parser
definition.  Consequently, the same convergence holds $P_\eta$-almost
surely as a statement about the output sequence.
\end{lemma}

\begin{proof}
Let
\[
  N_n:=\sum_{i=1}^{n-1}\one_{\{X_iX_{i+1}=00\}}.
\]
By synchronization, $N_n$ is the number of codeword starts whose first
two symbols are visible in $X_1^n$.  At most the last such start can
belong to a codeword that is not completely contained in the prefix.
After removing that possible incomplete word and discarding the first
complete word, one has the deterministic bound
\[
  |D_n-N_n|\le2.
\]
Define on the tower
\[
  f(\omega):=\one_{\{F(\omega)=0,\ F(S\omega)=0\}}.
\]
Then $f$ is bounded and measurable,
$\int f\,d\widehat\mu_\eta=P_\eta([00])=1/L$ by
\eqref{eq:start-intensity}, and
\[
  N_n=\sum_{i=0}^{n-2} f\circ S^i.
\]
By Lemma~\ref{lem:ergodic}, $S$ is ergodic.  Therefore
Theorem~\ref{thm:birkhoff}, first with denominator $n-1$ and then using
$(n-1)/n\to1$, gives
\[
  \frac{N_n}{n}\longrightarrow\frac1L
  \qquad\widehat\mu_\eta\text{-almost surely}.
\]
The bounded difference proves the assertion for $D_n$.
\end{proof}

\begin{proposition}\label{prop:weak}
The test $(\varphi_n)$ is weakly asymptotically consistent for $H_0$
against $H_1$.
\end{proposition}

\begin{proof}
Fix $P_\eta\in\cP$ and abbreviate $q_k=q_k(P_\eta)$.  Put
\[
  m_n:=\left\lfloor\frac{n}{2L}\right\rfloor.
\]
Since $m_n/n\to1/(2L)<1/L$, Lemma~\ref{lem:count} implies
\begin{equation}\label{eq:D-large}
  \widehat\mu_\eta(D_n<m_n)\longrightarrow0.
\end{equation}
Apply the parser to a genuine tower output.  Because the bit
$b$ does not affect a codeword's length, the complete parsed-word
boundaries, the usable base-time indices, and their marker values are
functions of the marker-suspension state $(K,j)$ alone.  Define
$\bar G_n\subset\widehat\Omega_K$ to be the event that at least one
usable parsed word has marker $k_n$.  Define the measurable function
\[
 \bar J_n:\widehat\Omega_K\to\N_0
\]
to be $0$ on $\bar G_n^c$ and, on $\bar G_n$, to be the base-time
index of the first usable word with marker $k_n$.  Measurability follows
directly from the finite deterministic parsing algorithm: for a fixed
$n$, at most $n$ codeword starts can meet coded times
$0,\ldots,n-1$, and every decision uses only finitely many coordinates
of $(K,j)$.  In particular,
$\bar J_n\in\{1,\ldots,n\}$ on $\bar G_n$.

On the full tower set
\[
 G_n:=\Omega_Y\times\bar G_n,
 \qquad
 J_n(y,K,j):=\bar J_n(K,j).
\]
Thus $G_n$ and $J_n$ are explicitly typed objects on
$\widehat\Omega=\Omega_Y\times\widehat\Omega_K$.  On
$\{D_n\ge m_n\}\cap G_n^c$, the first $m_n$ candidate markers all
differ from $k_n$.  Lemma~\ref{lem:marker-regeneration} therefore yields
\begin{align}
  \widehat\mu_\eta(G_n^c)
  &\le \widehat\mu_\eta(D_n<m_n)
       +(1-\pi_{k_n})^{m_n}\notag\\
  &\le \widehat\mu_\eta(D_n<m_n)
       +\exp(-m_n2^{-k_n})
  \longrightarrow0.                                  \label{eq:seen}
\end{align}
Indeed, for all sufficiently large $n$,
$k_n=\lfloor\tfrac12\log_2n\rfloor$, so
$2^{-k_n}\ge n^{-1/2}$ and
$m_n2^{-k_n}\to\infty$.

By \eqref{eq:factor-lift}, the marker-suspension coordinate $(K,j)$ is
independent of the environment $y$.  The bit output by the test on
$G_n$ is exactly $y_{J_n}^{k_n}$.  Since $\bar J_n$ has finite range
and $\eta$ is stationary,
\begin{align}
 \widehat\mu_\eta(G_n\cap\{y_{J_n}^{k_n}=1\})
 &=\sum_{t=1}^{n}
   \lambda(\bar G_n\cap\{\bar J_n=t\})\,
   \eta\{y:y_t^{k_n}=1\}\notag\\
 &=q_{k_n}\lambda(\bar G_n)
  =q_{k_n}\widehat\mu_\eta(G_n).             \label{eq:random-index}
\end{align}
Consequently,
\[
  \widehat\mu_\eta(\varphi_n=1)
  =q_{k_n}\widehat\mu_\eta(G_n)\le q_{k_n}.
\]
If $P_\eta\in H_0$, then $k_n\to\infty$ and the right-hand side tends
to $0$.

If $P_\eta\in H_1$, then
\[
  \widehat\mu_\eta(\varphi_n=0)
  =\widehat\mu_\eta(G_n^c)
   +(1-q_{k_n})\widehat\mu_\eta(G_n)
  \le \widehat\mu_\eta(G_n^c)+1-q_{k_n}\longrightarrow0
\]
by \eqref{eq:seen}.  Output-measurable events have the same probability
under $\widehat\mu_\eta$ and $P_\eta$, proving the proposition.
\end{proof}

\section{No strongly asymptotically consistent test exists}
\label{sec:no-strong}

\subsection{A marker-tail estimate independent of the environment}

For a marker-suspension state $(K,j)$, define the coded-time start
of the current codeword by $\tau_0:=-j$, and recursively define
\[
  \tau_{t+1}:=\tau_t+\ell(K_t),\qquad t\ge0.
\]
Thus codeword $t$ occupies coded times
$\tau_t,\tau_t+1,\ldots,\tau_{t+1}-1$ and has marker $K_t$.
Coded time $0$ lies in codeword $0$, because
$\tau_0=-j\le0<\ell(K_0)-j=\tau_1$.  The preceding codeword, of base
index $-1$, ends at time $\tau_0-1<0$, and therefore no codeword with a
negative base index can intersect a nonnegative coded time.
For $n,b\in\N$, let $R(n,b)\subset\widehat\Omega_K$ be the
measurable event that $K_t\le b$ for every $t\ge0$
whose codeword interval intersects $\{0,1,\ldots,n-1\}$.  Explicitly,
\[
 R(n,b)=\bigcap_{t\ge0}
 \left(
   \{\tau_t>n-1\}\cup\{\tau_{t+1}-1<0\}\cup\{K_t\le b\}
 \right),
\]
so measurability follows because each $\tau_t$ is a finite sum of
coordinate-measurable functions.  Since
$\tau_1=\ell(K_0)-j\ge1$ and every codeword has positive length,
$\tau_t\ge t$ for $t\ge1$.  Therefore every codeword intersecting
these $n$ coded times has base index in $\{0,1,\ldots,n\}$.  Hence
\[
  R(n,b)^c
  \subset
  \{K_0>b\}\cup\bigcup_{t=1}^{n}\{K_t>b\}.
\]
By \eqref{eq:length-biased} and the union bound,
\begin{equation}\label{eq:tail-bound}
  \lambda(R(n,b)^c)
  \le \frac1L\sum_{k>b}\pi_k\ell(k)
      +n\sum_{k>b}\pi_k.
\end{equation}
For every fixed $n$, the right-hand side tends to zero as
$b\to\infty$, because $\sum_k\pi_k\ell(k)=L<\infty$ and
$\sum_k\pi_k=1$.

\subsection{Recursive construction against an arbitrary test}

The sets $H_0$ and $H_1$ were defined in
Section~\ref{sec:family} without reference to any test.  Now let
\[
  \Gamma_n:\{0,1\}^n\to\{0,1\},\qquad n\in\N,
\]
be an arbitrary test.  Assume, for a contradiction, that it is
strongly asymptotically consistent for this already fixed pair
$(H_0,H_1)$.

At stage $s$, the provisional alternative $Q_s$ first selects a favorable
time $n_s$; the marker cutoff $b_s$ then shields that prefix from all
later coordinate blocks, and the new chain $Z^s$ is finally made
persistent enough to imitate $Q_s$ at that time with high conditional
probability.

Set
\[
  \varepsilon_s:=2^{-s-6},
  \qquad p_s:=\frac1{s+1},
  \qquad s\in\N,
\]
and $n_0:=0$.  We recursively construct consecutive finite intervals
\[
  I_s=[a_s,b_s]\cap\N,
  \qquad 1=a_1\le b_1<a_2\le b_2<\cdots,
  \qquad a_{s+1}:=b_s+1,
\]
strictly increasing times $n_s$, and independent stationary mixing
binary Markov chains $Z^s=(Z_t^s)_{t\in\Z}$.

Suppose $I_r$ and $Z^r$ have been defined for $r<s$.  By
consecutiveness, $I_1\sqcup\cdots\sqcup I_{s-1}=\{1,\ldots,a_s-1\}$.
For $s=1$, the product of the previously constructed chains is
understood as the one-point probability-preserving system and the
preceding union is empty.  Define the provisional environment
\[
  y_t^{+,s,k}:=
  \begin{cases}
    Z_t^r,&k\in I_r\text{ for some }r<s,\\
    1,&k\ge a_s.
  \end{cases}
\]
Let $\eta_s^+$ be its law.  It is an environment law: it is a
shift-commuting factor of the finite product of the mixing chains
$Z^1,\ldots,Z^{s-1}$, with deterministic coordinates adjoined, so
Lemma~\ref{lem:mixing-products} applies.  Let
\[
  Q_s:=P_{\eta_s^+}.
\]
For every $k\ge a_s$, $q_k(Q_s)=1$; hence $Q_s\in H_1$.
Strong consistency under $Q_s$ gives
\[
  \Gamma_n(X_1^n)\longrightarrow1
  \qquad Q_s\text{-almost surely}.
\]
Apply Theorem~\ref{thm:bounded-convergence} to the indicators
$\one_{\{\Gamma_n(X_1^n)=1\}}$, which converge almost surely to
$1$ and are bounded by $1$.  It follows that
$Q_s\{\Gamma_n=1\}\to1$.  Choose $n_s>n_{s-1}$ such that
\begin{equation}\label{eq:Qs}
  Q_s\{\Gamma_{n_s}=1\}>1-\varepsilon_s.
\end{equation}
By \eqref{eq:tail-bound}, choose $b_s\ge a_s$ such that
\begin{equation}\label{eq:Rs}
  \lambda(R(n_s,b_s)^c)<\varepsilon_s,
\end{equation}
and set $I_s=[a_s,b_s]\cap\N$.

Choose $\delta_s\in(0,1)$ sufficiently small that
\begin{equation}\label{eq:persist}
  (1-\delta_s)^{n_s}>1-\varepsilon_s
\end{equation}
and
\begin{equation}\label{eq:transition-bound}
  \alpha_s:=\frac{p_s\delta_s}{1-p_s}<1.
\end{equation}
Both inequalities hold for every sufficiently small positive
$\delta_s$.  Consider the two-state transition matrix, in the state
order $(0,1)$,
\[
  M_s=
  \begin{pmatrix}
    1-\alpha_s&\alpha_s\\
    \delta_s&1-\delta_s
  \end{pmatrix}.
\]
Its stationary distribution is $\nu_s=(1-p_s,p_s)$ because
$(1-p_s)\alpha_s=p_s\delta_s$.  Both off-diagonal entries and both
diagonal entries lie in $(0,1)$, so the chain is irreducible and
aperiodic.  More explicitly, the two eigenvalues of $M_s$ are $1$ and
$\lambda_s:=1-\alpha_s-\delta_s\in(-1,1)$.  Let $\Pi_s:=\mathbf 1\nu_s$, where $\mathbf 1$ is the
two-dimensional column vector of ones.  Direct multiplication gives
$\Pi_s^2=\Pi_s$, $M_s\Pi_s=\Pi_sM_s=\Pi_s$, and
\[
 M_s=\Pi_s+\lambda_s(I-\Pi_s).
\]
Since $\Pi_s(I-\Pi_s)=(I-\Pi_s)\Pi_s=0$, induction gives
\[
  M_s^m=\Pi_s+\lambda_s^m(I-\Pi_s)
       =\mathbf 1\nu_s+\lambda_s^m(I-\mathbf 1\nu_s).
\]  Hence
\[
  M_s^m(a,b)\longrightarrow\nu_s(b)
  \qquad(a,b\in\{0,1\}).
\]
Lemma~\ref{lem:bilateral-markov} therefore supplies a unique stationary
bilateral Markov law with transition matrix $M_s$ and marginal
$\nu_s$, and its bilateral shift is mixing.  Let $Z^s$ have this law,
independently of all previously constructed chains.  Such independent
realizations exist on the product of their canonical path spaces.
This completes the recursive step; induction therefore defines all
$I_s$, $n_s$, and $Z^s$.

The intervals are finite and nonempty because
$a_s\le b_s<\infty$; they are consecutive and disjoint because
$a_{s+1}=b_s+1$.  Inductively, $a_s\ge s$, so $a_s\to\infty$.  For any
$k\in\N$, the set $\{s:a_s\le k\}$ is finite and nonempty; let $r$ be
its maximum.  Then $k<a_{r+1}=b_r+1$, so $a_r\le k\le b_r$ and
$k\in I_r$.  Thus
\begin{equation}\label{eq:partition-N}
  \N=\bigsqcup_{s\ge1}I_s.
\end{equation}

\subsection{The final null environment is mixing}

Let
\[
  \Omega_Z:=\prod_{s\ge1}\{0,1\}^{\Z},
  \qquad
  \zeta:=\bigotimes_{s\ge1}\law(Z^s).
\]
This countable product law exists by
Theorem~\ref{thm:kolmogorov-extension}: for each finite set of pairs
$(s,t)\in\N\times\Z$, take the product, over the finitely many
indices $s$ involved, of the corresponding finite-dimensional Markov
laws.  These finite-dimensional laws are projectively consistent and
make the coordinate path processes $Z^s$ mutually independent with the
prescribed laws.
Let the product shift act coordinatewise.  By
Lemma~\ref{lem:mixing-products}, $\zeta$ is mixing.  Define the
measurable shift-commuting map $\Psi:\Omega_Z\to\Omega_Y$ by
\[
  \bigl(\Psi((z^s)_{s\ge1})\bigr)_t^k:=z_t^s
  \qquad\text{when }k\in I_s,
\]
where $s$ is unique by \eqref{eq:partition-N}.  Each output
coordinate of $\Psi$ is a coordinate projection, so $\Psi$ is
measurable; it also satisfies $\Psi\circ\sigma_Z=\sigma\circ\Psi$,
where $\sigma_Z$ is the coordinatewise product shift.  Put
\[
  \eta^*:=\Psi_\#\zeta.
\]
By Lemma~\ref{lem:mixing-products}(b), $\eta^*$ is mixing; the
pushforward and intertwining identities make it shift-invariant.  Thus
$\eta^*$ is an environment law.

Let
\[
  P^*:=P_{\eta^*}.
\]
For $k\in I_s$,
\[
  q_k(P^*)=q_k(\eta^*)=\Prob(Z_0^s=1)=p_s.
\]
If $k\to\infty$ and $k\in I_{s(k)}$, then necessarily
$s(k)\to\infty$ because every finite union
$I_1\cup\cdots\cup I_r$ is bounded.  Since $p_s\to0$,
\begin{equation}\label{eq:Pstar-H0}
  q_k(P^*)\longrightarrow0,
  \qquad\text{so }P^*\in H_0.
\end{equation}

\subsection{Coupling the final null with the provisional alternatives}

Place all random objects on the product probability space
\[
  (\Omega_Z\times\widehat\Omega_K,
   \zeta\otimes\lambda),
\]
and denote its probability by $\widetilde\Prob$.  Let $(K,j)$ be the
marker-suspension coordinate.  From the chain coordinates construct
the final environment $y^*=\Psi((Z^s)_s)$ and, using the common
$(K,j)$, construct its coded output $X^*$.  By
\eqref{eq:factor-lift}, $X^*$ has law $P^*$.

For each $s$, construct on the same space the provisional environment
\[
  y_t^{+,s,k}:=
  \begin{cases}
    Z_t^r,&k\in I_r\text{ for some }r<s,\\
    1,&k\ge a_s,
  \end{cases}
\]
and use the same $(K,j)$ to construct its coded output $X^{+,s}$.
The coordinates $Z^1,\ldots,Z^{s-1}$ have exactly the independent
product law used to define $\eta_s^+$ and are independent of
$(K,j)\sim\lambda$; hence $X^{+,s}$ has law $Q_s$.

Define
\[
  E_s:=\{Z_t^s=1\text{ for every }t=0,1,\ldots,n_s\},
\]
\[
  A_s^*:=\{\Gamma_{n_s}((X^*)_1^{n_s})=1\},
  \qquad
  A_s^+:=\{\Gamma_{n_s}((X^{+,s})_1^{n_s})=1\}.
\]
The event $R(n_s,b_s)$ was defined on $\widehat\Omega_K$.  On the
present coupling space we use its explicit pullback
\[
  R_s:=\Omega_Z\times R(n_s,b_s)
  =\{(z,(K,j)):(K,j)\in R(n_s,b_s)\}.
\]

The events $E_s$ are independent, because each depends on a distinct
chain.  Stationarity and the Markov property give
\begin{equation}\label{eq:E-prob}
  \widetilde\Prob(E_s)
  =p_s(1-\delta_s)^{n_s}
  >p_s(1-\varepsilon_s).
\end{equation}
Since $\varepsilon_s<1/2$ and
$\sum_s p_s=\infty$, one has
$\sum_s\widetilde\Prob(E_s)=\infty$.  Part~(b) of Theorem~\ref{thm:borel-cantelli} therefore yields
\begin{equation}\label{eq:E-io}
  \widetilde\Prob(E_s\ \io)=1.
\end{equation}

We next prove exact prefix agreement on $E_s\cap R_s$.  In the shared
suspension state, coded time $0$ lies in base word $0$, and every word
with a negative base index ends before coded time $0$, as established
above.  Also $\tau_t\ge t$ for $t\ge1$.  Hence every codeword
intersecting coded times $0,\ldots,n_s-1$ has a base index
$t\in\{0,\ldots,n_s\}$.  On $R_s$, its marker $k$ is at most $b_s$.
If $k<a_s$, then $k\in I_r$ for some $r<s$, and the final and
provisional environments both use $Z_t^r$.  If
$a_s\le k\le b_s$, then $k\in I_s$; on $E_s$, the final environment
uses $Z_t^s=1$, while the provisional environment also uses $1$.
Thus every codeword contributing a symbol to the first $n_s$ outputs
is identical in the two constructions, and
\begin{equation}\label{eq:prefix-agreement}
  (X^*)_1^{n_s}=(X^{+,s})_1^{n_s}
  \qquad\text{on }E_s\cap R_s.
\end{equation}
It follows that
\[
  (A_s^*)^c\cap E_s
  \subset
  (R_s^c\cap E_s)\cup((A_s^+)^c\cap E_s).
\]

The event $E_s$ is measurable with respect to the path coordinate
$Z^s$ alone.  The event $R_s$ is measurable with respect to $(K,j)$
alone, and $A_s^+$ is measurable with respect to
$(K,j),Z^1,\ldots,Z^{s-1}$.  These coordinate sigma-fields are
independent under $\zeta\otimes\lambda$.  Therefore both $R_s$ and
$A_s^+$ are independent of $E_s$.  Since $X^{+,s}$ has law $Q_s$ and
$\widetilde\Prob(E_s)>0$,
\begin{align}
 \widetilde\Prob((A_s^*)^c\mid E_s)
 &\le \widetilde\Prob(R_s^c\mid E_s)
      +\widetilde\Prob((A_s^+)^c\mid E_s)\notag\\
 &=\lambda\bigl(R(n_s,b_s)^c\bigr)
      +Q_s\{\Gamma_{n_s}=0\}
 <2\varepsilon_s,                                  \label{eq:conditional-error}
\end{align}
using \eqref{eq:Qs} and \eqref{eq:Rs}.  Hence
\begin{align*}
 \sum_{s\ge1}\widetilde\Prob(E_s\cap(A_s^*)^c)
 &=\sum_{s\ge1}\widetilde\Prob(E_s)
   \widetilde\Prob((A_s^*)^c\mid E_s)\\
 &\le\sum_{s\ge1}2\varepsilon_s
 <\infty.
\end{align*}
By part~(a) of Theorem~\ref{thm:borel-cantelli},
\[
  \widetilde\Prob(E_s\cap(A_s^*)^c\ \io)=0.
\]
Together with \eqref{eq:E-io}, this implies
\begin{equation}\label{eq:A-io}
  \widetilde\Prob(A_s^*\ \io)=1.
\end{equation}
Indeed, on almost every sample path, $E_s$ occurs infinitely often,
while only
finitely many of those occurrences can be accompanied by
$(A_s^*)^c$.

But $X^*$ has law $P^*\in H_0$.  Strong consistency of $\Gamma$ on
$H_0$ would imply that, almost surely, there is an $N$ such that
$\Gamma_n((X^*)_1^n)=0$ for every $n\ge N$.  Since $(n_s)$ is a strictly increasing sequence of integers, it tends
to infinity; this would force
$\widetilde\Prob(A_s^*\ \io)=0$, contradicting
\eqref{eq:A-io}.  Thus the arbitrary test $\Gamma$ cannot be strongly
asymptotically consistent for $(H_0,H_1)$.

Proposition~\ref{prop:weak} and the preceding contradiction prove
Theorem~\ref{thm:main}.

\section{Conclusion}

The pair $(H_0,H_1)$ is fixed independently of any candidate strong
test.  The construction therefore gives a counterexample to the
unrestricted asymptotic-consistency implication in Ryabko's
Conjecture~5.1: weak asymptotic consistency does not, in general, imply
strong asymptotic consistency for arbitrary disjoint families of
stationary ergodic process distributions.  The result does not address
variants in which additional regularity conditions are imposed on the
hypothesis families.

The weak test succeeds because, at sample size $n$, it requires only
one observation from a marked coordinate $k_n\to\infty$, whose marginal
mean converges to the label of the corresponding hypothesis.  In
contrast, every putative strongly consistent test is defeated by a
process in the fixed null family constructed from independent slowly
switching Markov chains whose states are replicated across successive
coordinate blocks.  The stationary one-probabilities
$p_s\downarrow0$ are chosen to be nonsummable, while the one-phases are
made sufficiently persistent.  These phases occur infinitely often
and cause the process to imitate suitable members of $H_1$ along an
unbounded sequence of sample sizes, even though its intrinsic
coordinate means converge to zero and its law belongs to $H_0$.

\end{document}